\documentclass[12pt]{amsart}
\usepackage{graphicx}
\usepackage[top=1in,bottom=1in,left=1.25in,right=1.25in]{geometry} 
\usepackage{amsmath,amssymb,amsfonts}
\usepackage{amsthm}
\usepackage{color}
\usepackage{hyperref}
\newtheorem{theo}{Theorem}[section]

\newtheorem{lemma}[theo]{Lemma}

\newtheorem{cor}[theo]{Corollary}

\begin{document}
\title{A note on knots with H(2)-unknotting number one}
\author{Yuanyuan Bao}
\address{
Department of Mathematics,
Tokyo Institute of Technology,
Oh-okayama, Meguro, Tokyo 152-8551, Japan
}
\email{bao.y.aa@m.titech.ac.jp}
\date{}
\begin{abstract}
We give an obstruction to unknotting a knot by adding a twisted band, derived from Heegaard Floer homology.
\end{abstract}
\thanks{The author is supported by scholarship from
the Ministry of Education, Culture, Sports, Science and Technology of Japan.}
\maketitle

\section{introduction}
Many unknotting operations have been defined and studied in knot theory. For example, as well-known, (a), (b) (cf. \cite{MR915115, MR995383}) and (c) in Figure~\ref{fig:fig1} are three types of unknotting operations. Especially, (c) was introduced by Hoste, Nakanishi and Taniyama \cite{MR1075165}, which they called H($n$)-move. Here $n$ is the number of arcs inside the circle. Note that an H($n$)-move is required to preserve the component number of the diagram. The H($n$)-unknotting number of a knot is the minimal number of H($n$)-moves needed to change the knot into the unknot. In this note, we focus on the special case when $n$ equals two. Given two knots $K$ and $K'$, when $K'$ is obtained from $K$ by applying an H(2)-move, we also alternatively say that $K'$ is obtained from $K$ by adding a twisted band, as shown in Figure~\ref{fig:fig4}. We only choose those bands for which the diagrams before and after represent knots. Following \cite{MR1075165}, we denote the H(2)-unknotting number of a knot $K$ by $u_{2}(K)$. In this note, we give a necessary condition for a knot $K$ to have $u_{2}(K)=1$, by using a method introduced by Ozsv{\'a}th and Szab{\'o} \cite{MR2136532}. 

The question whether a given knot has H(2)-unknotting number one should be traced back to Riley. He made the conjecture that the figure-eight knot could never be unknotted by adding a twisted band. Lickorish confirmed this conjecture in \cite{MR859958}. Here we give a brief review of his method. Given a knot $K$, let $\Sigma (K)$ denote the double-branched cover of $S^{3}$ along $K$ and let $\lambda : H_{1}(\Sigma (K))\times H_{1}(\Sigma (K)) \rightarrow{\mathbb Q}/{\mathbb Z}$ be the linking form of $\Sigma(K)$. Lickorish proved that if the knot $K$ can be unknotted by adding a twisted band, then $H_{1}(\Sigma (K))$ is cyclic and it has a generator $g$ such that $\lambda(g,g)=\pm 1/{\rm det}(K)$, where ${\rm det}(K)$ is the determinant of $K$.
For the figure-eight knot $4_{1}$, the linking form has the form $\lambda(g,g)=2/5$ for some generator $g\in H_{1}(\Sigma (4_{1}))\cong {\mathbb Z}/5{\mathbb Z}$. If there is another generator $g'=xg$ such that $\lambda(g',g')=\pm 1/5$, we have $2x^{2}\equiv \pm 1 \pmod{5}$. There is no such an integer $x$ satisfing the condition. Therefore Riley's conjecture holds.

\begin{figure}[h]
	\centering
		\includegraphics[width=0.6\textwidth]{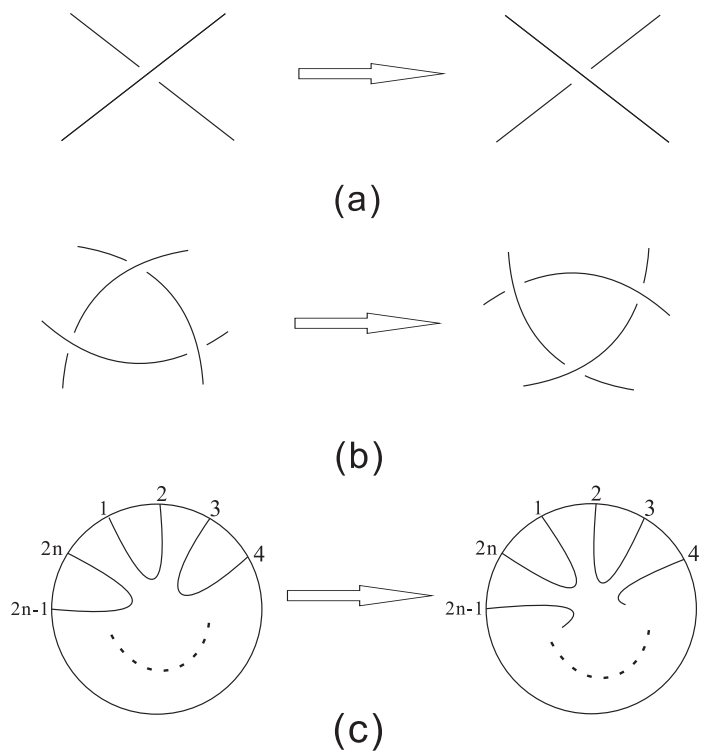}
	\caption{Some unknotting operations.}	
	\label{fig:fig1}
\end{figure}

\begin{figure}
	\centering
		\includegraphics[width=0.55\textwidth]{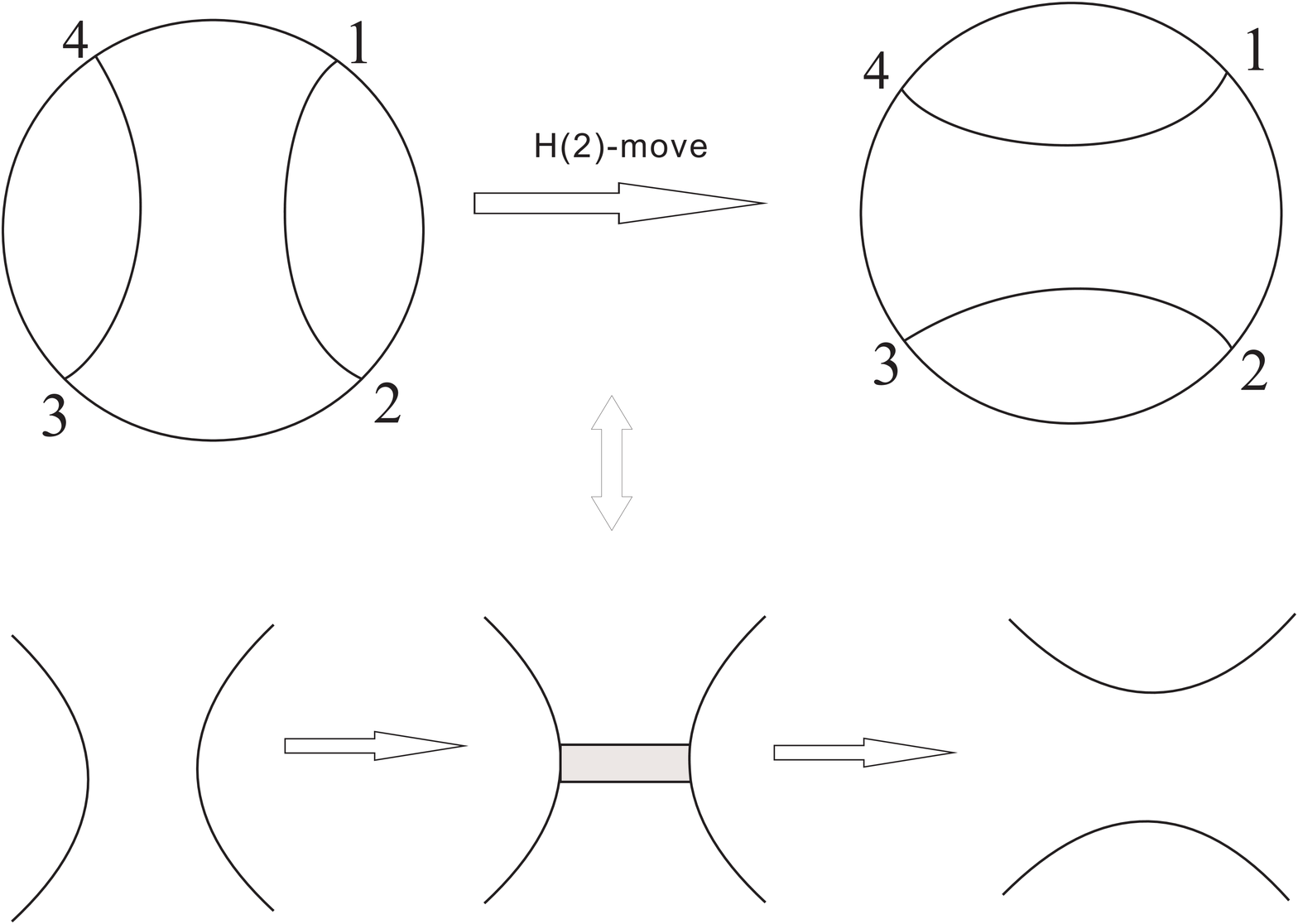}
		\caption{Adding a twisted band to a knot diagram.}
	\label{fig:fig4}
\end{figure}

Now we turn to the description of our result.
Consider a positive-definite symmetric $n\times n$ matrix $Q$ over ${\mathbb Z}$. Suppose ${\rm det}(Q)$ is $p$. Then $Q$ as a presentation determines a group $G$. A characteristic vector for $Q$ is an element in 

\begin{equation*}
\begin{split}
{\rm char}(Q)&=\left\{\xi\in {\mathbb Z}^{n} | \xi^{t}v\equiv v^{t}Qv \pmod{2} \text{ for any $v\in {\mathbb Z}^{n}$}\right\}\\
&=\left\{\xi\in {\mathbb Z}^{n} | \xi_{i}\equiv Q_{ii} \pmod{2}\right\}.
\end{split}
\end{equation*}

Two characteristic vectors $\xi$ and $\zeta$ are said to be equivalent if $Q^{-1}(\xi-\zeta)\in {\mathbb Z}^{n}$. Suppose $p$ is odd, and consider the map (cf. \cite{MR2388097, MR2136532})
\begin{equation*}
M_{Q}: G \longrightarrow {\mathbb Q}
\end{equation*} 
defined by 
\begin{equation*}
M_{Q}(\alpha)=\min\left\{\left.\frac{\xi^{t}Q^{-1}\xi-n}{4}\right|\xi\in {\rm char}(Q), [\xi]=\alpha\in G\right\}. 
\end{equation*}
The map is well-defined up to an automorphism of $G$. 

Now we recall the definition of Goeritz matrix. Given a knot diagram, color this diagram in checkerboard fashion such that the unbounded region has black color. Let $f_{0}, f_{1}, \ldots, f_{k}$ denote the black regions and $f_{0}$ correspond to the unbounded one. Define the sign of a crossing as in Figure~\ref{fig:fig2}. Then the Goeritz matrix $Q$ is the $k\times k$ symmetric matrix defined as follows
\begin{equation}
q_{ij}=
\begin{cases}
\text{the signed count of crossings adjacent to $f_{i}$} & \text{if $i=j$},\\
\text{minus the signed count of crossings joining $f_{i}$ and $f_{j}$} & \text{if $i\neq j$}
\end{cases}
\end{equation}
for $i,j=1,2,\ldots, k$.

\begin{figure}
	\centering
		\includegraphics[width=0.6\textwidth]{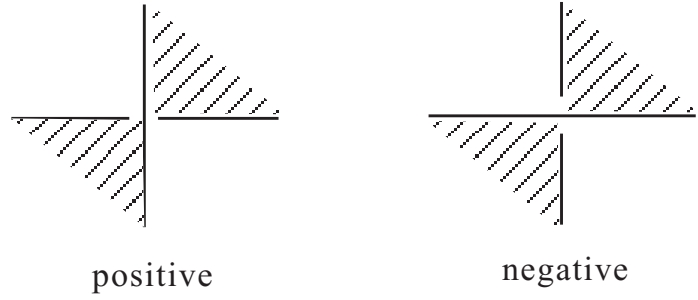}
		\caption{The sign convention of a crossing.}
	\label{fig:fig2}
\end{figure}

Our result about H(2)-unknotting number is as follows:
\begin{theo}
\label{main}
Let $K$ be an alternating knot with determinant $p$, and let $Q$ be the positive-definite Goeritz matrix corresponding to a reduced alternating diagram of $K$ or its mirror image. Suppose $G$ is the group presented by $Q$. If $u_{2}(K)=1$, then there is an isomorphism $\phi : {\mathbb Z}/|p|{\mathbb Z}\longrightarrow G$ and a sign $\epsilon \in \{+1,-1\}$ with the properties that for all $i\in {\mathbb Z}/|p|{\mathbb Z}$:
$$I_{\phi, \epsilon}(i):=\epsilon\cdot M_{Q}(\phi(i))-\frac{1}{4}(\frac{1}{|p|}(\frac{|p|+(-1)^{i}|p|}{2}-i)^{2}-1)=0 \pmod{2},$$
$$\text{and}\quad I_{\phi, \epsilon}(i)\leq 0.$$
\end{theo}

If one is familar with the work in \cite{MR2136532}, the proof is immediate. We will give the proof in Section 2.
We study the H(2)-unknotting number of the pretzel knot $P(13,4,11)$ as an example, to show that the obstruction obtained here works better than other ones that the author knows.
\begin{cor}
\label{pretzel}
The pretzel knot $P(13,4,11)$ has H(2)-unknotting number 2.
\end{cor}




\section{Proofs}
\subsection{Proof of Theorem~\ref{main}}
Given a 3-manifold $Y$ and one of its ${\rm spin}^{c}$-structures $s$, an invariant $d(Y, s)$ called correction term is defined for the pair $(Y, s)$ in \cite{MR1957829}. Suppose $Y$ is an oriented rational homology sphere. When $|H^{2}(Y,{\mathbb Z})|$ is odd, there exists a canonical isomorphism between the space ${\rm Spin}^{c}(Y)$ of ${\rm spin}^{c}$-structures on $Y$ and $H^{2}(Y,{\mathbb Z})$. In this case, we replace $s$ in $d(Y, s)$ by the corresponding element in $H^{2}(Y,{\mathbb Z})$. Ozsv{\'a}th and Szab{\'o} studied knots with unknotting number one in \cite{MR2136532}, and here is an general result they obtained (also refer to \cite{MR2388097}).
\begin{theo}[Ozsv{\'a}th-Szab{\'o}\cite{MR2136532}]
\label{os1}
Let $Y$ be a rational homology 3-sphere which is the boundary of a simply-connected positive-definite four-manifold $W$, with $H^{2}(Y,{\mathbb Z})$ of odd order. If the intersection form of $W$ is represented in a basis by the matrix $A$ and $G_{A}$ is the group presented by $A$, then there exists a group isomorphism $\phi: G_{A}\rightarrow H^{2}(Y,{\mathbb Z})$ with
\begin{equation}
\begin{split}
d(Y, \phi (\alpha))&\leq M_{A}(\alpha)\\
and \quad d(Y, \phi (\alpha))&\equiv M_{A}(\alpha) \pmod{2}
\end{split}
\end{equation}
for all $\alpha\in G_{A}$.
\end{theo}

When $K$ is an alternating knot in $S^{3}$, the correction terms for $\Sigma (K)$ have an extremely easy combinatorial description as follows.
\begin{theo}[Ozsv{\'a}th-Szab{\'o}\cite{MR2136532, MR2141852}]
\label{os2}
If $K$ is an alternating knot and $Q$ denotes a Goeritz matrix associated to a reduced alternating projection of $K$, and $G$ is the group presented by $Q$, then there is an isomorphism $\varphi: G\rightarrow H^{2}(\Sigma(K), {\mathbb Z})$, with the property that $$d(\Sigma (K), \varphi (\alpha))= M_{Q}(\alpha)$$ for all $\alpha\in G$.
\end{theo}

\begin{proof}[Proof of Theorem \ref{main}]
If the H(2)-unknotting number of $K$ is equal to one, then by Montesinos's trick \cite{MR0380802} we have $\Sigma (K)=\epsilon\cdot S^{3}_{|p|}(C)$ for some knot $C\subset S^{3}$ and $\epsilon \in \{+1, -1\}$. Here $p$ is equal to ${\rm det}(K)$. The manifold $-S^{3}_{|p|}(C)$ represents the manifold with reversed orientation. Therefore $\epsilon\cdot\Sigma (K)=S^{3}_{|p|}(C)$ bounds a four-manifold $W$, which is obtained by attaching a 2-handle to a four-ball along $C$ with framing $|p|$. The intersection form of $W$ is $A=(|p|)$. In this case we have that $G_{A}={\mathbb Z}/|p|{\mathbb Z}$, that $W$ is a simply-connected 4-manifold and that $H^{2}(S_{|p|}^{3}(C),{\mathbb Z})\cong{\mathbb Z}/|p|{\mathbb Z}$.

By Theorem~\ref{os1}, there exists a group isomorphism $\phi: {\mathbb Z}/|p|{\mathbb Z} \rightarrow H^{2}(S_{|p|}^{3}(C),{\mathbb Z})$ with
\begin{equation}
\begin{split}
d(\epsilon\cdot \Sigma (K), \phi (i))=\epsilon\cdot d(\Sigma (K), \phi (i))&\leq M_{A}(i)\\
{\rm and} \quad \epsilon\cdot d(\Sigma (K), \phi (i))&\equiv M_{A}(i) \pmod{2}
\end{split}
\end{equation}
for all $i\in {\mathbb Z}/|p|{\mathbb Z}$. It is easy to check that $M_{A}(i)=\frac{1}{4}(\frac{1}{|p|}(\frac{|p|+(-1)^{i}|p|}{2}-i)^{2}-1)$. Now Theorem~\ref{main} follows from Theorem~\ref{os2}.

\end{proof}

\begin{figure}[h]
	\centering
		\includegraphics[width=0.6\textwidth]{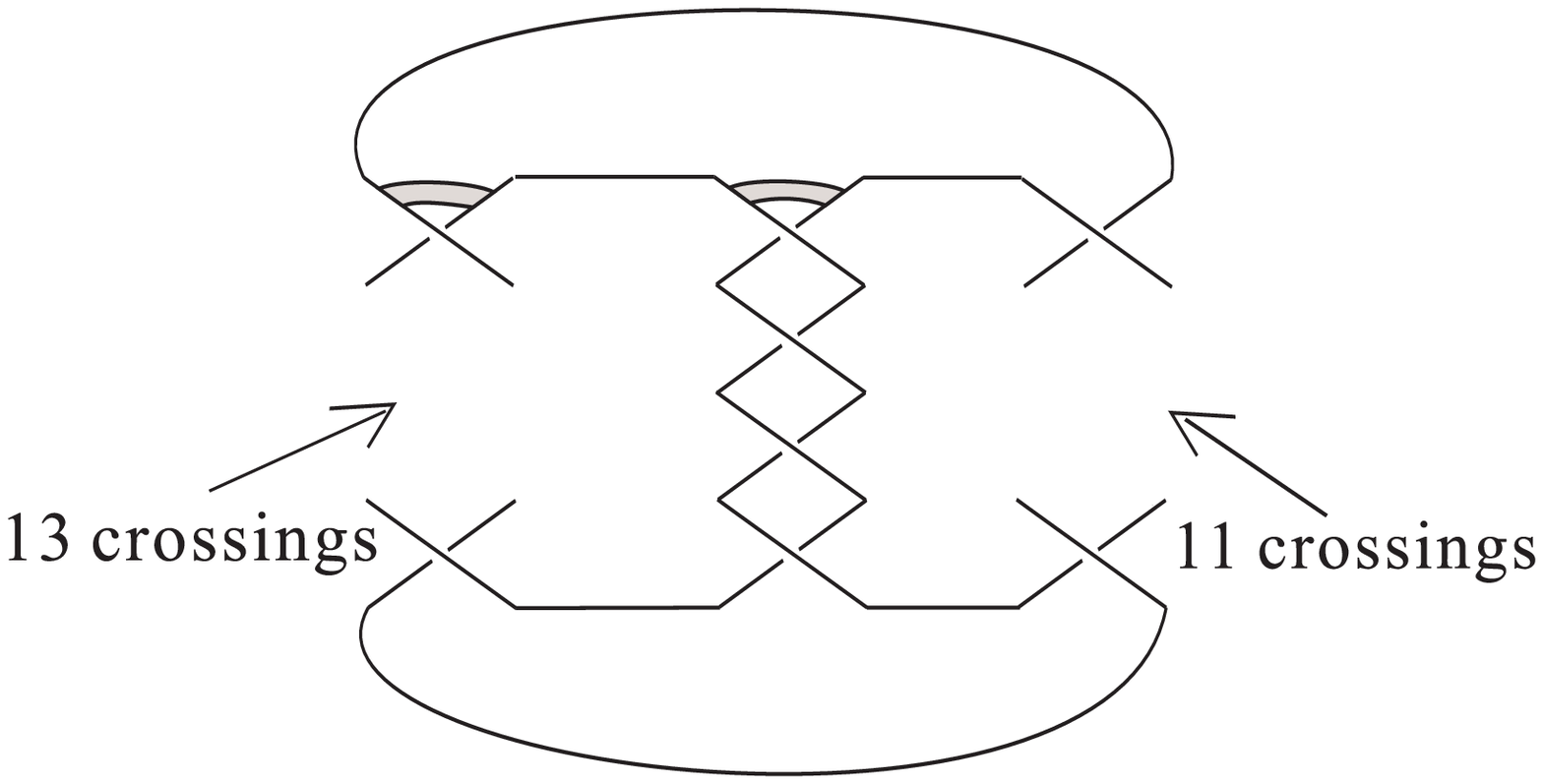}
		\caption{The pretzel knot $P(13,4,11)$.}
	\label{fig:fig3}
\end{figure}

\subsection{An example}
The pretzel knot $K=P(13,4,11)$ is a knot as shown in Figure~\ref{fig:fig3}. A Goeritz matrix associated to this diagram is 
$$Q=
\left( \begin{array}{ccccc}
17       & -4   \\
-4        & 15  
\end{array} \right) 
,$$ and the determinant is ${\rm det}(Q)={\rm det}(K)=239$. Suppose $G$ is the group presented by $Q$. In fact, the group $G$ is isomorphic to ${\mathbb Z}/239{\mathbb Z}$. In the following calculation, we take the vector $(0,1)^{t}$ as a generator of $G$. By calculation, it is easy to see that for any isomorphism $\phi : {\mathbb Z}/239{\mathbb Z}\longrightarrow {\mathbb Z}/239{\mathbb Z}$ there is $$I_{\phi, \epsilon}(0)=\epsilon\cdot M_{Q}(\phi(0))-119/2=(\epsilon\cdot11-119)/2.$$ Since $I_{\phi, \epsilon}(0)$ has to be an even number, therefore we have $\epsilon=+1$. Next we obtain that $I_{\phi, +1}(1)=M_{Q}(\phi(1))+119/478$. To guarantee that $I_{\phi, +1}(1)$ is an even number, the isomorphism $\phi$ has to be either $\phi_{1}=15$ or $\phi_{2}=224$. By calculation, we see that $I_{\phi_{1}, +1}(1)=I_{\phi_{2}, +1}(1)=4$, a positive number, which conflicts with the necessary condition stated in Theorem~\ref{main}. Therefore the H(2)-unknotting number of $P(13,4,11)$ has to be at least two. On the other hand, the knot $P(13,4,11)$ can be changed into the unknot by adding two twisted bands as shown in Figure~\ref{fig:fig3}. Hence the H(2)-unknotting number of $P(13,4,11)$ is two. This completes the proof of Corollary~\ref{pretzel}.

\subsection{Comparisons with other criterions}

There have been many criterions and properties which can be used to bound the H(2)-unknotting number of a knot. We want to apply them to the knot $P(13,4,11)$ and compare the results with Corollary~\ref{pretzel}.

The first one is Lickorish's obstruction that we recalled in the beginning. But it does not work for the pretzel knot $K=P(13,4,11)$. It is known that the Goeritz matrix $Q$ is a presentation of $H_{1}(\Sigma(K),{\mathbb Z})$, and $Q^{-1}$ represents the linking form $\lambda$. From Section 2.2, we known that $I_{\phi_{1}, +1}(1)$ is an integer. This implies that $\lambda(g,g)=1/239$ over ${\mathbb Q}/{\mathbb Z}$ for $g=(0,15)^{t}$. The vector $g$ can work as a generator of $H_{1}(\Sigma(K),{\mathbb Z})$.


There are two invariants of knots which are closely related to H(2)-unknotting number. Given a knot $K \subset S^{3}$, the crosscap number \cite{MR1362987} of $K$ is defined as follows:
$$\gamma(K)=\min\left\{\beta_{1}(F)\left|\text{$F$ is a non-orientable connected surface in $S^{3}$ and $\partial F=K$}\right.\right\}.$$
The four-dimensional crosscap number of $K$ \cite{MR1690998}, which we denote $\gamma^{*}(K)$ here, is by name defined as follows:
\begin{equation*}
\gamma^{*}(K)=\min\left\{\beta_{1}(F)\left|
\begin{split}
&\text{$F$ is a non-orientable connected smooth surface in $B^{4}$ and}\\
& \partial F=K\subset \partial B^{4}=S^{3}
\end{split}
\right.\right\}.
\end{equation*}
Their relation with H(2)-unknotting number is as follows. We give a proof here since we have not found any reference of it.
\begin{lemma}
\label{relation}
Given a knot $K\subset S^{3}$, we have $\gamma^{*}(K)\leq u_{2}(K) \leq \gamma(K)$.
\end{lemma}
\begin{proof}
The knot $K$ can be reconstructed from the unknot by adding $u_{2}(K)$ twisted bands successively. Let $D$ be a disk bounded by the unknot and $b_{1},b_{2},\ldots,b_{u_{2}(K)}$ be the bands added to the boundary of $D$. Then $F:=D\cup \bigcup_{i=1}^{u_{2}(K)}b_{i}$ is a non-orientable surface in $B^{4}$ with $\partial F=K$. We have $\gamma^{*}(K)\leq \beta_{1}(F)=u_{2}(K)$. The second inequality is proved as follows. Suppose $S$ is a non-orientable surface in $S^{3}$ which realizes the crosscap number of $K$. Namely we have $\beta_{1}(S)=\gamma(K)$ and $\partial S=K$. Then there are $\gamma(K)$ disjoint essential arcs in $S$, say $\tau_{1},\tau_{2},\cdots,\tau_{\gamma(K)}$, such that $S- \tau_{i}$ has one boundary component for $i=1,2,\cdots,\gamma(K)$ and $S-\bigcup_{i=1}^{\gamma(K)}\tau_{i}$ is a disk. If we add twisted bands to $K$ along $\tau_{i}$ for $i=1,2,\cdots,\gamma(K)$, the resulting knot is the unknot. Therefore we have $u_{2}(K) \leq \gamma(K)$.
\end{proof}
Ichihara and Mizushima \cite{MR2556097} calculated the crosscap numbers of pretzel knots. According to their calculation, the crosscap numbers of $P(13,4,11)$ is two, but the four-dimensional crosscap number of it is unknown. Therefore the H(2)-unknotting number of $P(13,4,11)$ cannot be determined by Lemma~\ref{relation} so far. Kanenobu and Miyazawa \cite{MR2573402} introduced some criterions for bounding the H(2)-unknotting number of a knot, but their methods cannot be applied to the knot $P(13,4,11)$, either.

\bibliographystyle{siam}
\bibliography{bao}

\end{document}